\documentclass[10pt,a4paper,leqno]{amsart}
\usepackage{amsmath}
\usepackage{amssymb}
\usepackage{amsfonts}
\usepackage{amsthm}
\usepackage{hyperref}
\usepackage{mathrsfs}
\usepackage{dsfont}
\usepackage{mathtools}
\usepackage{tikz-cd}
\usepackage{esint}
\usepackage{color}

\newcommand*{\mailto}[1]{\href{mailto:#1}{\nolinkurl{#1}}}

\definecolor{darkgreen}{rgb}{0.5,0.25,0}
\definecolor{darkblue}{rgb}{0,0,0.56}
\definecolor{answerblue}{rgb}{0,0,0.75}

\hypersetup{colorlinks,breaklinks,
            linkcolor=darkblue,urlcolor=darkblue,
            anchorcolor=darkblue,citecolor=darkblue}
            
\allowdisplaybreaks

\newcommand{\ep}{\varepsilon}
\newcommand{\eps}{\varepsilon}

\newcommand{\pd}{\partial}

\renewcommand{\d}{\mathrm{d}}

\newcommand{\Px}{\mathbb{P}}
\newcommand{\sgn}{\mathrm{sgn}}

\newcommand{\bk}[1]{\left(#1 \right)}

\newcommand{\R}{\mathbb{R}}
\newcommand{\N}{\mathbb{N}}

\newcommand{\abs}[1]{\left | #1 \right |}
\newcommand{\norm}[1]{\left \| #1 \right \|}

\newcommand{\bS}{\mathcal{S}}

\DeclareMathOperator*{\Ex}{\mathbb{E}}

\newtheorem{thm}{Theorem}[section]
\newtheorem{prop}[thm]{Proposition}
\newtheorem{lem}[thm]{Lemma}

\theoremstyle{definition}

\theoremstyle{remark}
\newtheorem{rem}{Remark}[section]
\numberwithin{equation}{section}

\title[SDEs with irregular random drift]
{Strong solutions of a stochastic differential equation
with irregular random drift}

\author[Holden]{Helge Holden}
\address[Helge Holden]{Department of Mathematical Sciences\\
  NTNU Norwegian University of Science and Technology\\
  NO-7491 Trondheim\\ Norway}
\email{\mailto{helge.holden@ntnu.no}}
\urladdr{\url{https://www.ntnu.edu/employees/holden}}

\author[Karlsen]{Kenneth H. Karlsen}
\address[Kenneth H.\ Karlsen]{Department of Mathematics \\ 
University of Oslo \\
P.O.\ Box 1053, NO-0316 Oslo\\
Norway}
\email{\mailto{kennethk@math.uio.no}}

\author[Pang]{Peter H.C. Pang}
\address[Peter H.C. Pang]{Department of Mathematical Sciences\\
  NTNU Norwegian University of Science and Technology\\
  NO-7491 Trondheim\\ Norway}
\email{\mailto{peter.pang@ntnu.no}}

\subjclass[2010]{60H10, 34F05}

\keywords{Stochastic differential equation, random drift, 
irregular drift, one-sided gradient bound, strong solution, 
well-posedness, existence, uniqueness}

\thanks{This research was partially supported 
by the Research Council of Norway Toppforsk 
project {\em Waves and Nonlinear Phenomena (WaNP)} (250070).}

\date{\today}

\begin{document}

\begin{abstract}
We present a well-posedness result for  
strong solutions of one-dimensional stochastic 
differential equations (SDEs) of the form
$$
\d X= u(\omega,t,X)\, \d t
+ \frac12 \sigma(\omega,t,X)\sigma'(\omega,t,X)\,\d t
+ \sigma(\omega,t,X) \, \d W(t),
$$
where the drift coefficient $u$ is random 
and irregular. The random and regular noise coefficient 
$\sigma$ may vanish. The main contribution is 
a pathwise uniqueness result under the assumptions that $u$ 
belongs to $L^p(\Omega; L^\infty([0,T];\dot{H}^1(\R)))$ 
for any finite $p\ge 1$, $\Ex\abs{u(t)-u(0)}_{\dot{H}^1(\R)}^2
\to 0$ as $t\downarrow 0$, and $u$ satisfies the 
one-sided gradient bound 
$\partial_x u(\omega,t,x) \le K(\omega, t)$, 
where the process $K(\omega,t )>0$ 
exhibits an exponential moment bound of the 
form $\Ex \exp\bigl(p\int_t^T K(s)\,\d s\bigr) 
\lesssim {t^{-2p}}$ for small times $t$, 
for some $p\ge1$.  This study is motivated by 
ongoing work on the well-posedness of the 
stochastic Hunter--Saxton equation, a stochastic 
perturbation of a nonlinear transport equation that 
arises in the modelling of the director 
field of a nematic liquid crystal. 
In this context, the one-sided bound acts as a selection 
principle for dissipative weak 
solutions of the stochastic partial 
differential equation (SPDE).
\end{abstract}

\maketitle


\section{Introduction}
\subsection{Main result}
In this paper, we prove strong existence 
and pathwise uniqueness for a class of one-dimensional SDEs 
with rough random drift $u=u(\omega,t,x)$ 
and a noise coefficient $\sigma=\sigma(\omega,t,x)$ that is 
random and possibly degenerate. We fix 
a stochastic basis $\bS=\bk{\Omega, \mathcal{F}, 
\{\mathcal{F}_t\}_{t \ge 0}, \mathbb{P}}$ 
consisting of a complete probability space
$\left(\Omega, \mathcal{F},\Px\right)$ 
and a complete right-continuous filtration 
$\{\mathcal{F}_t\}_{t \ge 0}$. Moreover, we fix 
a standard Brownian motion $W$ on $\bS$ 
adapted to the filtration $\{\mathcal{F}_t\}_{t \ge 0}$. 

We are interested in strong solutions $X$, i.e., 
$\mathbb{P}$-almost surely continuous and 
$\{\mathcal{F}_t\}_{t \ge 0}$-adapted stochastic 
processes $X$ satisfying
\begin{align}\label{eq:X_sde_a}
	\d X =u(\omega,t,X) \, \d t 
	+\frac14\left(\sigma^2\right)'(\omega,t,X)\,\d t
	+\sigma(\omega,t, X)\, \d W, 
	\quad X(0) = x \in \R,
\end{align}
where $\sigma'$ denotes the $x$-derivative 
of $\sigma=\sigma(\omega,t,x)$, 
so that $\frac14\left(\sigma^2\right)'
=\frac12 \sigma \partial_x\sigma$. 
For a deterministic, sufficiently regular 
$\sigma=\sigma(x)$, the SDE \eqref{eq:X_sde_a} 
can be written as
\begin{equation}\label{eq:X_sde_b}
	\d X = u(\omega,t,X)\,\d t + \sigma(X)\circ \d W,
\end{equation}
where $\circ$ denotes the Stratonovich differential.

The random non-smooth drift $u$ is an 
$\{\mathcal{F}_t\}_{t \ge 0}$-progressively measurable 
process that belongs to 
$L^p(\Omega;L^\infty([0,T];\dot{H}^1(\R)))$, 
for $p\in [1,\infty)$. The semi-normed vector 
space $\dot{H}^1(\R)$ is defined as the 
subspace of functions in $L^\infty(\R)$ 
having a weak derivative in $L^2(\R)$, with 
semi-norm $\abs{h}_{\dot{H}^1(\R)}
=\norm{\partial_x h}_{L^2(\R)}$.
Note that this ensures that $u$ is 
$\frac12$-H\"older continuous 
in $x$, which is not enough for uniqueness. We 
additionally assume that $u$ satisfies the 
following one-sided gradient bound:
\begin{equation}\label{eq:stoch_oleinik_g}
q(\omega,t,x) := 
		\pd_x u(\omega,t,x) 
		\le K(\omega,t), 
		\quad \text{where $K>0$,}
\end{equation}
and, for some $p>1$, 
\begin{equation}\label{eq:log_divergence2}
		\Ex \exp\bk{p \int_\ep^T K(s)\,\d s}
		\lesssim_{p,T} \ep^{-2p}, 
		\quad \text{for all  $\eps\in (0,1)$}.
\end{equation}
Here we use the notation $h_1\lesssim_{\alpha}h_2$ if $h_1\le C(\alpha) h_2$ for some constant
$C$ that may depend on $\alpha$, and non-negative functions $h_1,h_2$.
Finally, we require a strong temporal 
continuity condition at $t=0$: 
\begin{equation}\label{eq:u_one-sided_continuity}
	\lim_{t \downarrow 0} \Ex 
	\abs{u(t) - u(0)}_{\dot{H}^1(\R)}^2 = 0.
\end{equation}

The conditions imposed on the drift $u$ 
are motivated by the work \cite{MR4151173}, 
in which $u$ solves a nonlinear stochastic transport 
equation, and the one-sided gradient 
bound \eqref{eq:stoch_oleinik_g} acts as a selection 
principle for dissipative weak solutions of this SPDE. 
We will return to the motivation 
behind the key condition \eqref{eq:stoch_oleinik_g} later. 

Regarding the noise coefficient $\sigma$, let 
us discuss the  case of a deterministic $\sigma$ first, cf.~\eqref{eq:X_sde_b}. 
In this case, we assume that 
$\sigma=\sigma(x)$ satisfies 
\begin{align}\label{eq:sigma_bound}
	\sigma \in C^2(\R), 
	\quad \sigma',\sigma'', 
	\bk{\sigma^2}''\in L^\infty(\R).
\end{align} 
For such a $\sigma$, which is necessarily globally 
Lipschitz continuous and of linear growth, the second derivative
$\frac12\bk{\sigma^2}'' = \bk{\sigma'}^2+\sigma \sigma''$ 
is bounded on $\R$. An example ensuring the latter 
is when $\sigma'$, $\sigma''$ are bounded and $\sigma''$ 
is compactly supported on $\R$; then 
$\bk{\sigma^2}''\lesssim 1$. The noise 
coefficient $\sigma$ is allowed to vanish in this work. 

The main contribution of this paper 
is the treatment of the irregular random drift $u$. 
However, it turns out that our methods are 
sufficiently flexible to allow for a wider 
class of random noise coefficients 
$\sigma=\sigma(\omega,t,x)$. The conditions 
defining this class appear somewhat eloborate, but 
any deterministic $\sigma = \sigma(x)$ 
satisfying \eqref{eq:sigma_bound} belongs to this class. 
First, we require that $\sigma=\sigma(\omega,t,x)$ is 
progressively measurable (on $\bS$), and that 
$\sigma$, $\bk{\sigma^2}'$ are globally $x$-Lipschitz 
in the sense that
\begin{equation}
	 \abs{\sigma(\omega,t,x)-\sigma(\omega,t,y)}, 
	\, \abs{\bk{\sigma^2}'(\omega,t,x)
	-\bk{\sigma^2}'(\omega,t,y)} 
	\le \Lambda(\omega,t)\abs{x-y},
	\label{eq:sigma_lip}
\end{equation}
where the progressively measurable 
	process $\Lambda(t)=\Lambda(\omega,t)$
exhibits exponential moments,
\begin{equation}
	\Ex \exp \bk{p\int_0^T \Lambda^2(t)\,\d t}
	<\infty, \,\, \forall p>0.
	\label{eq:sigma_exp}
\end{equation}
The exponential moments \eqref{eq:sigma_exp} are used to 
prove the \textit{existence} of a solution $X$ that 
belongs (locally) to $L^p(\Omega; C([0,T]))$ for any finite $p\ge 1$. 
Dropping the requirement of arbitrary $p$-moments, 
one can relax \eqref{eq:sigma_exp} somewhat. 

By Jensen's inequality, the 
condition \eqref{eq:sigma_exp} implies
\begin{equation}\label{eq:sigma_lipschitz}
	\Ex \int_0^T \Lambda^2(t) \,\d t < \infty,
\end{equation}
which will be used on certain occasions. Finally, we will 
also need the following technical conditions:
\begin{subequations}\label{eq:sigma_alle}
\begin{align}
	&\norm{\bk{\sigma^2}''(0)}_{L^\infty(\Omega \times \R)} 
	< \infty;
	\label{eq:sigma_init}
	\\&
	\bk{\sigma^2}''(t)-\bk{\sigma^2}''(0) 
	\in L^2(\Omega \times \R),
	\quad \text{uniformly on $[0,T]$};
	\label{eq:sigma_L2_incl}
	\\ &
	\lim_{t \downarrow 0}
	\Ex\norm{\bk{\sigma^2}''(t) 
	-\bk{\sigma^2}''(0) }_{L^2(\R)}^2 = 0;
	\label{eq:sigma_one-sided_tcont}
	\\ &
	\Ex \int_0^T \abs{\sigma(t,0)}^p\,\d t < \infty,
	\quad 
	\Ex \int_0^T \abs{\left(\sigma^2\right)'(t,0)}^p 
	\,\d t<\infty, \quad \forall p\in [1,\infty).
	\label{eq:sigma_origin}
\end{align}
\end{subequations}
\begin{rem}
It is possible to consider $\sigma=\sigma(\omega,t,x)$ 
such that $\bk{\sigma^2}''$ satisfies the same conditions
\eqref{eq:stoch_oleinik_g} and \eqref{eq:u_one-sided_continuity}
as $q$. In this case, for our existence result, we 
must additionally assume \eqref{eq:sigma_lipschitz}.
However, these conditions will fail to include 
the linear case $\sigma(x)=a+bx$, 
which originally motivated this study 
(see Section \ref{sec:motivation}).
\end{rem}

Our main result is the following theorem.

\begin{thm}\label{thm:main}
Suppose $u \in L^p(\Omega;L^\infty([0,T];
\dot{H}^1(\R)))$ satisfies 
conditions \eqref{eq:stoch_oleinik_g} 
and \eqref{eq:u_one-sided_continuity}, and 
$\sigma$ satisfies \eqref{eq:sigma_lip}, \eqref{eq:sigma_exp},  
\eqref{eq:sigma_init}--\eqref{eq:sigma_origin}. 
There exists a unique strong solution 
of \eqref{eq:X_sde_a}.
\end{thm}

The central part of Theorem \ref{thm:main} is 
the uniqueness assertion (cf.~Theorem~\ref{thm:pathwise_uniqueness}). 
We prove pathwise uniqueness by a careful estimation 
of the difference between two solutions, making essential 
use of the Tanaka formula, the exponential moment bound \eqref{eq:stoch_oleinik_g}, and 
a recent stochastic Gronwall inequality \cite{MR4058986,MR1799274} 
(see Lemma \ref{lem:stochastic_gronwall} below). 
The exponential bound \eqref{eq:stoch_oleinik_g}, along 
with \eqref{eq:u_one-sided_continuity}, 
allows us to control the difference between the two 
solutions for short times $t\le \eps$ ($\eps \ll 1$), 
which is the main challenge in demonstrating pathwise uniqueness. 
When $\sigma\equiv 0$, our uniqueness result 
recovers \cite[Prop.~A]{MR1799274}. 
The detailed proof reported in 
Section \ref{sec:pathwise_unique} 
can be viewed as a surprisingly non-trivial stochastic 
extension of the ODE proof in \cite{MR1799274}.
 
In Section \ref{sec:existence}, we demonstrate 
existence of strong solutions to the SDE \eqref{eq:X_sde_a} 
(cf.~Theorem~\ref{thm:existence}). We 
approximate \eqref{eq:X_sde_a} using ``one-sided 
truncations" $\{u_R\}$ of the drift $u$, and then make use of 
Krylov's theorem \cite{MR1731794} for SDEs with random 
coefficients to solve \eqref{eq:X_sde_a} 
with $u=u_R$. This produces a family of solutions 
$\{X_R\}$, indexed by the truncation level $R$ with $R\to \infty$. 
We show that $\{X_R\}$ constitutes a Cauchy sequence in the space 
$L^{1/2}(\Omega;C([0,T]))$, with metric 
$d(X_1,X_2) := \Ex \sup_{t \in [0,T]} 
\abs{X_1(t)-X_2(t)}^{1/2}$ 
(cf.~Proposition \ref{prop:cauchy_seq}) \cite[4.7.62]{Bog2007}. 
The proof of this result proceeds along 
the lines of the uniqueness argument.
The Cauchy property, along with \eqref{eq:sigma_exp} and 
$R$-independent $p$-moments of $X_R$ 
(cf.~Lemma \ref{lem:X_R_pmoment}), implies the 
existence of a limit $X\in L^2(\Omega; C([0,T])$ 
such that $X_R\to X$ in $L^2(\Omega; C([0,T])$. 
It is straightforward to deduce that $X$ 
is a solution of \eqref{eq:X_sde_a} 
(cf.~Theorem \ref{thm:existence}).
  
Before discussing the literature on SDEs 
with irregular drift and the motivation behind 
our particular class of drift coefficients $u$, let us 
supply a relevant example of the process 
$K$ arising in \eqref{eq:stoch_oleinik_g}.

\begin{rem}\label{rem:example}
Consider the SDE \eqref{eq:X_sde_b} with 
deterministic $\sigma = \sigma(x)$ 
satisfying \eqref{eq:sigma_bound}. 
According to  Section \ref{sec:motivation} below, 
it makes sense to impose the condition
\begin{align}\label{eq:stoch_oleinik}
	q(\omega,t,x)= 
	\pd_x u(\omega,t,x) 
	\le  C+\frac{e^{-\norm{\sigma'}_{L^\infty} W(t)}}
	{\frac12\int_0^t e^{-\norm{\sigma'}_{L^\infty}W(s)}\,\d s},
\end{align}
where $C\ge 0$ is a constant. Let us verify 
that $q$ satisfies \eqref{eq:stoch_oleinik_g}. 
With
$$
K(\omega, t) := C+\frac{e^{-\norm{\sigma'}_{L^\infty} W(t)}}
{\frac12\int_0^t e^{-\norm{\sigma'}_{L^\infty} W(s)}\,\d s} 
= C+2\frac{\d}{\d t} \log\bk{\frac12\int_0^t 
e^{-\norm{\sigma'}_{L^\infty}W(s)}\,\d s},
$$
we find that $I(\eps):=\Ex \exp\left( p 
\int_\ep^T K(\omega,s)\,\d s\right)$ satisfies
$$
I(\eps)=e^{pC(T - \ep)}\Ex\bk{\frac{\int_0^T 
\exp\bk{-\norm{\sigma'}_{L^\infty} W(s)}\,\d s}
{\int_0^\ep \exp\bk{-\norm{\sigma'}_{L^\infty}W(s)}
\,\d s}}^{2p}.
$$
By the Cauchy--Schwarz inequality, 
$I(\eps)\lesssim_{T,p} (I_-)^{1/2} 
(I_+)^{1/2}$, 
where
\begin{align*}
	&I_+:= \Ex\left(\int_0^T 
	\exp\bigl(-\norm{\sigma'}_{L^\infty} W(s)\bigr) 
	\,\d s\right)^{4p},
	\\
	& I_-:= \Ex\left(\int_0^{\eps} 	
	\exp\bigl(-\norm{\sigma'}_{L^\infty} W(s)\bigr)
	\,\d s\right)^{-4p}.
\end{align*}
We estimate $I_-$ as follows:
\begin{align*}
	I_- & \le  \Ex \bk{\ep \min_{s \in [0,\ep]}
	\exp\bk{-\norm{\sigma'}_{L^\infty}W(s)}}^{-4p}
	\\ & \le  \Ex \bk{\ep \exp\bk{-\norm{\sigma'}_{L^\infty} 
	\max_{s \in [0,\ep]} W(s)}}^{-4p}
	\\ & = \frac{2}{\sqrt{2 \pi \ep}}
	\int_0^\infty  \ep^{-4p} 
	\exp\big(4p \norm{\sigma'}_{L^\infty} x 
	-\frac{x^2}{2\ep}\big)\,\d x
	\lesssim_{p,\sigma} \ep^{-4p},
\end{align*}
where we have used that the law on $[0,\infty)$ of 
$\max_{s \in [0,\ep]} W(s)$ is equivalent to the law 
of $\abs{W(\ep)}$ \cite[Prop.~III.3.7]{MR1303781}, 
for which we have
$$
\mathbb{P}\bigl(\abs{W(t)} \in \d x\bigr) 
= \frac{2}{\sqrt{2 \pi t}} 
e^{-x^2/(2t)}\,\d x.
$$ 
Similarly, $I_+\lesssim_{\sigma,T,p} 1$. 
Hence $I(\eps) \lesssim_{p,T,\sigma} \ep^{-2p}$ (for all 
finite $p$), i.e., \eqref{eq:stoch_oleinik_g} holds.
\end{rem}

\subsection{Background}
Let us contextualise our result by discussing 
some previous studies on the well-posedness of SDEs. 
There is a very rich literature studying the existence 
and uniqueness of solutions, which begins 
with It\^{o}'s work on SDEs with globally Lipschitz 
coefficients (see \cite[Chap.~IX]{MR1303781}).
Often the Lipschitz condition is too strong. 
While weak existence is relatively 
easy to obtain for non-smooth coefficients 
(via, say, Girsanov's theorem), the construction 
of strong solutions is a more delicate matter. 
Strong solutions of SDEs with rough deterministic 
coefficients have been studied by many authors, 
beginning with \cite{MR0336813,Veretennikov:1980ut}, and later 
\cite{MR1392450,MR1864041,MR2117951,MR2810591}, 
to mention just a few examples. Most of these works use 
the Fokker--Planck PDE associated with the SDE, the Krylov 
estimate, and the Zvonkin transformation, which require the 
noise coefficient to be non-degenerate (uniformly elliptic). 
As a consequence, the results hold under 
very weak conditions on the drift, much weaker 
than in deterministic ODEs. For recent work 
on the well-posedness of SDEs with (Sobolev) rough coefficients 
and degenerate noise, see \cite{MR3785594}. 
A probabilistic approach based on Malliavin 
calculus (nondegenerate noise) is developed 
in \cite{MR2368369,MR3096525}.
Most of the cited articles assume additive noise.
The works \cite{MR2820071,MR2172887} 
consider multiplicative noise  under 
non-degeneracy and Sobolev regularity conditions on 
the noise coefficient. For a detailed study of 
one-dimensional SDEs, see the book \cite{MR2112227}.

The influential paper \cite{MR2593276} studied stochastic 
regularisation in linear transport SPDEs 
with non-smooth velocity $b$, for 
which the \textit{characteristic equation} is
\begin{align}\label{eq:FGP_SDE}
	\d X = b(t,X)\,\d t + \d W.
\end{align}
Using the It\^o--Tanaka trick and solution regularity 
of the associated Fokker--Planck 
equation (a backward parabolic equation), they establish 
uniqueness of solutions to stochastic transport equations 
under a regularity condition on $b$ that is weaker 
than in the DiPerna--Lions--Ambrosio 
theory of deterministic transport equations.
To do so they prove the existence and uniqueness of 
solutions to the SDE \eqref{eq:FGP_SDE} with minimal regularity 
assumptions on $b$ using the short-time smooth 
flow of the associated backward parabolic equation.  
In \cite[Sec.~6.2]{MR2593276} they give 
negative examples showing that their results 
do not hold for equations with random drift $b$, a typical 
example of which is $b=b(\omega,t,x)=\abs{x-W(t)}^{1/2}\wedge 1$.
Whilst this $b$ is locally in $\dot{H}^1(\R)$, 
$\pd_x b$ does not satisfy a one-sided bound 
of the form \eqref{eq:stoch_oleinik_g}. Motivated by \cite{MR2593276}, 
there were many additional works studying strong solutions 
of SDEs like \eqref{eq:FGP_SDE} with non-smooth 
drift $b$, but almost all of them assume 
that $b$ is deterministic.

Let us turn our attention to SDEs with random coefficients.
In \cite{MR1731794}, Krylov established the existence 
and uniqueness of strong solutions to 
\begin{equation}\label{eq:random-SDE}
\d X = b(\omega,t,X) \,\d t 
+ \sigma(\omega,t,X) \,\d W,
\end{equation}
under some boundedness, monotonicity, and coercivity 
conditions on the random coefficients $b$ and $\sigma$. 
His proof is based on a detailed convergence analysis of the 
Euler discretization scheme. We state
Krylov's result as Theorem \ref{thm:krylov_existence} 
below, and use it in Section \ref{sec:existence}
as a part of the existence proof. Because of an 
indispensable ``logarithmic divergence" at $t=0$, 
Krylov's theorem does not apply to the SDE \eqref{eq:X_sde_a} 
with $u$ satisfying the one-sided 
gradient bound \eqref{eq:stoch_oleinik_g}. 

With a random drift $b$ and $\sigma\equiv 1$ 
in \eqref{eq:random-SDE}, the work \cite{MR3565418} 
partially recovered the results of \cite{MR2593276} 
under an additional condition of Malliavin differentiability 
of $b$. The proof employed a Girsanov transformation 
idea \cite{MR0336813}, which 
extends the It\^o--Tanaka trick in \cite{MR2593276}, by 
considering a backward parabolic SPDE instead 
of the Fokker--Planck PDE associated with $X$ 
for a deterministic $b$. We also refer 
to \cite{MR4021273} for a related 
result, which allows for the drift 
$b(\omega,t,x)=b_1(t,x)+b_2(\omega,t,x)$, 
where the deterministic part $b_1$ is measurable 
and of linear growth. 
In contrast, the random part $b_2$ is sufficiently smooth 
in $t,x$ and Malliavin differentiable in $\omega$.
These results were extended and sharpened 
in \cite{Zha2020} to the SDE \eqref{eq:random-SDE} 
with non-degenerate noise and random coefficients 
$b$ and $\sigma$ satisfying similar $(t,x)$-regularity 
and Malliavin differentiability conditions. 
An illustrative example of random drift $b$ covered 
by these recent works is $b(\omega,t,x)=f(t,x,W(t))$ 
for a function $f$ that is Lipschitz continuous 
in the last variable. The works 
\cite{MR3565418,MR4021273,Zha2020} 
cannot handle the SDE \eqref{eq:random-SDE} with random drift 
$u\in L^p(\Omega; L^\infty([0,T];\dot{H}^1(\R)))$ satisfying
\eqref{eq:stoch_oleinik_g} and \eqref{eq:u_one-sided_continuity}, 
even if we were to assume that $\sigma(\cdot)>0$. The proof 
of our Theorem~\ref{thm:main} will 
not use ideas based on the associated backward SPDE, nor will 
we impose non-degeneracy or Malliavin 
differentiability conditions on our coefficients.

\subsection{Motivation}\label{sec:motivation}
We conclude this introduction with a brief motivation of 
the current study, which stems from our ongoing 
investigation into the uniqueness and dissipation properties 
of solutions to the stochastic 
Hunter--Saxton equation \cite{MR4151173}
\begin{align}\label{eq:HS1}
	\d q + \pd_x \left(u q\right) \,\d t 
	- \frac12 q^2 \, \d t 
	+ \pd_x \left(\sigma q \right) \circ \d W=0, 
	\qquad \pd_x u =q.
\end{align}
Existence results, along with a specific distribution for 
wave-breaking (finite-time blowup and continuation), 
were derived for the nonlinear transport-type SPDE \eqref{eq:HS1} 
in \cite{MR4151173}. These results were derived under the 
condition that $\sigma$ is linear. 
Solutions to \eqref{eq:HS1} were constructed 
from its characteristic equation, namely the SDE 
\eqref{eq:X_sde_a}. 

Using the It\^o--Wentzell 
theorem and the characteristic equation \eqref{eq:X_sde_b}, 
the following Lagrangian formulation of \eqref{eq:HS1} 
can be postulated:
\begin{align}\label{eq:lagrangian_HS}
	\d \mathfrak{Q} = -\frac{1}{2} \mathfrak{Q}^2 \,\d t
	-\sigma' \mathfrak{Q}\circ \d W, 
	\qquad \mathfrak{Q}(0) = q(0,x).
\end{align}
This SDE can be solved exactly as a 
\textit{stochastic Verhulst equation}. 
The solution is
$$
\mathfrak{Q}(t,x) = 
\frac{e^{-\sigma' W(t)}}{\frac1{q(0,x)} 
+\frac12\int_0^t e^{-\sigma' W(s)}\,\d s}.
$$

In \cite{MR4151173}, we constructed the drift $u$ 
directly in such a way that it was obvious 
that \eqref{eq:X_sde_a} was well-posed, and 
$\mathfrak{Q}(t,x)=\pd_x u(t,X(t,x))$ 
solved \eqref{eq:lagrangian_HS}, providing 
us with a way to construct solutions to 
the stochastic Hunter--Saxton equation \eqref{eq:HS1} 
along characteristics. The solution to the 
SDE \eqref{eq:lagrangian_HS} identifies the 
dissipative solution of the SPDE \eqref{eq:HS1} 
with an Ole{\u\i}nik-type (one-sided gradient) 
bound. This motivates our study of the SDE \eqref{eq:X_sde_a}
with random drift $u$ satisfying \eqref{eq:stoch_oleinik}, 
and thus \eqref{eq:stoch_oleinik_g}.

In an ongoing work, we study the uniqueness 
question for the stochastic Hunter--Saxton 
equation \eqref{eq:HS1}. In that work, starting from a 
solution to the SPDE \eqref{eq:HS1}, 
we must derive properties of the solution to 
the characteristic equation \eqref{eq:X_sde_b}. 
The well-posedness theorem in the present paper, which 
we believe is of independent interest, is needed 
as a part of that endeavour.

\begin{rem}
Finally, we present an example of 
a random drift $u$ motivated 
by \eqref{eq:HS1}, cf.~\cite{MR4151173}. 
Fixing a number $c\in\R$, let $Z_1(t)$ 
be the unique solution to 
$$
Z_1(t)=\frac{c^2}2 \int_0^t Z_1(s)\,\d s
+\int_0^t c \, Z_1(s) \,\d W.
$$
Fixing a number $v_0>0$, we introduce
$$
Z_2(t)=Z_1(t)+\exp\left(c W(t)+\int_0^t 
\frac{\exp\bigl(-c W(s)\bigr)}{-v_0+\frac12 
\int_0^s \exp\bigl(-c W(r)\bigr)\,\d r}\,\d s\right).
$$
Finally, we set
$$
Z_3(t) = \bigl(Z_2(t) - Z_1(t)\bigr)
\frac{\exp\bigl(-c W(t)\bigr)}{-v_0+
\frac12 \int_0^t \exp\bigl(-c W(s)\bigr)\,\d s}. 
$$
Denote by $T^\star=T^\star(\omega)$ the 
(blow-up) time for which
$$
\lim_{t\uparrow T^\star}
\int_0^t \exp\bigl(-c W(s)\bigr)\,\d s
=2v_0.
$$
Now we define the adapted and continuous 
drift coefficient $u$ by
\begin{align*}
	u(\omega,t,x)=
	\frac{x - Z_1(t)}{Z_2(t) - Z_1(t)} Z_3(t) 
	\mathds{1}_{\left[0,T^\star\right)
	\times \left[Z_1(t),Z_2(t)\right)}
	+ Z_3(t) \mathds{1}_{\left[0,T^\star\right)
	\times \left[Z_2(t),\infty\right)}.
\end{align*}
Clearly, the gradient
$$
\pd_x u(t)= 
\frac{\exp\bigl(-c W(t)\bigr)}{-v_0+
\frac12 \int_0^t \exp\bigl(-c W(s)\bigr)\,\d s}
\mathds{1}_{\left[0,T^\star\right)
\times \left[Z_1(t),Z_2(t)\right)}
$$
blows up ($\pd_x u \to -\infty$ while $\abs{u}$ 
remains bounded) as $t\uparrow T^\star$ but 
evidently \eqref{eq:stoch_oleinik}, and thus 
\eqref{eq:stoch_oleinik_g}, holds. Besides,  
$\pd_x u \in L^p(\Omega; L^\infty([0,T]; \dot{H}^1(\R)))$ 
for all $p\ge 1$, and one can easily check that 
$u$ obeys \eqref{eq:u_one-sided_continuity}. 
Note that $u(t)\equiv 0$ for all $t>T^\star$, which 
corresponds to a dissipative solution of 
the stochastic Hunter--Saxton equation \eqref{eq:HS1}.
\end{rem}

\section{Pathwise uniqueness}\label{sec:pathwise_unique}
In this section, we prove the uniqueness part 
of Theorem \ref{thm:main}. We make essential
use of the stochastic Gronwall inequality 
established recently by Scheutzow \cite{MR3078830}. 
The proof in \cite{MR3078830} relies on a martingale 
inequality of Burkholder that holds for continuous martingales. 
Below we recall a mild refinement due to Xie and Zhang
 \cite[Lemma~3.8]{MR4058986} which holds for general 
discontinuous martingales. The stochastic Gronwall 
lemma provides an upper bound for 
the $p$th moment of a process $\xi$ that does not 
depend on the martingale part $M$ of the inequality. 
It is this convenient ``martingale uniformity" 
that forces $p\in (0,1)$.

\begin{lem}[\cite{MR4058986}]\label{lem:stochastic_gronwall}
Fix a stochastic basis $\bS$. Let $\xi(t)$ 
and $\eta(t)$ be non-negative adapted processes, 
$A(t)$ be a non-decreasing adapted process 
starting at $A(0) =0$, 
and $M$ be a local martingale with $M(0) = 0$. 
Suppose $\xi$ is c\`adl\`ag in time and satisfies 
the following pathwise differential inequality:
$$
\d \xi \le \eta \,\d t + \xi\,\d A + \d M 
\quad \text{on $[0,T]$}.
$$
For any $0 < p < r < 1$ and $t \in [0,T]$, 
\begin{align*}
	\bk{\Ex \sup_{s \in [0,t]} 
	\xi^p(s)}^{1/p} \le 
	C_{p,r} \bk{\Ex 
	\exp\bk{\frac{r}{1-r}A(t)}}^{(1 - r)/r} 
	\Ex \bk{\xi(0) + \int_0^t \eta(s) \,\d s},
\end{align*}
where $C_{p,r}=\bk{\frac{r}{r - p}}^{1/p}$. 
\end{lem}

We are now in a position to prove the following result.

\begin{thm}[Pathwise uniqueness]\label{thm:pathwise_uniqueness}
Suppose $u\in L^p(\Omega;L^\infty([0,T];\dot{H}^1(\R)))$ 
satisfies conditions \eqref{eq:stoch_oleinik_g} 
and \eqref{eq:u_one-sided_continuity}, and $\sigma$ 
satisfies \eqref{eq:sigma_lip}, 
\eqref{eq:sigma_lipschitz}--\eqref{eq:sigma_one-sided_tcont}. 
Let $X_1$ and $X_2$ be two (strong) solutions of 
the SDE \eqref{eq:X_sde_a} on $[0,T]$, 
with $T>0$ finite. Uniqueness holds in the following sense:
\begin{equation}\label{eq:pathwise-uniq}
	\Ex \sup_{t \in [0,T]} 
	\abs{X_2(t) - X_1(t)}^{1/2}=0.
\end{equation}
Consequently, 
$\Px\Bigl(\bigl\{\omega\in \Omega : 
X_1(\omega,t)=X_2(\omega,t)
\, \, \forall t\in [0,T]\bigr\} \Bigr)=1$,  
i.e., $X_1$ and $X_2$ are indistinguishable. 
\end{thm}

\begin{proof}
Let $X_1$, $X_2$, and $T$ be as in the statement of the 
theorem. Without loss of generality, we 
assume throughout the proof that
\begin{equation}\label{eq:N-bound}
	\abs{X_i(t)}\le N, \quad 
	t\in [0,T], \quad i=1,2,
\end{equation}
for some $N>0$. Indeed, introducing the stopping time
$$
\tau_N:=\inf\left\{t\in [0,T]: 
\abs{X_1(t)}>N \,\,\, \text{or} 
\,\,\, \abs{X_2(t)}>N\right\},
$$
we may replace $X_i$ by $\tilde{X}_i(t)
:=X_i\bigl(t\wedge \tau_N\bigr)$, which satisfies 
$\abs{\tilde{X}_i(t)}\le N$ for all $t\in (0,\tau_N]$. 
The SDE for $\tilde{X}_i$ becomes
\begin{align*}
	\tilde{X}_i(t)
	& = x+ \int_0^{t\wedge \tau_N}u(s,X_i(s))\,\d s 
	+\frac14 \int_0^{t\wedge \tau_N}
	\left(\sigma^2\right)'(s,X_i(s)) \,\d s 
	\\ &\quad 
	+\int_0^{t\wedge \tau_N}
	\sigma(s,X_i(s))\, \d W(s)
	\\ & = x+ \int_0^t 
	u\left(s,\tilde{X}_i(s)\right)\,\d s 
	+\frac14 \int_0^t
	\left(\sigma^2\right)'\left(s,\tilde{X}_i(s)\right) \,\d s 
	\\ &\quad +\int_0^t 
	\sigma\left(s,\tilde{X}_i(s)\right)\, \d W(s), 
	\quad t\in [0,\tau_N].
\end{align*}
We can therefore apply the upcoming argument to 
$\tilde{X}_2-\tilde{X}_1$ on $[0,\tau_N]$ 
instead of to $X_2-X_1$ on $[0,T]$, to deduce that
$$
\Ex \sup_{t \in [0,T]} \abs{X_2\left(t\wedge \tau_N\right) 
- X_1\left(t\wedge \tau_N\right)}^{1/2}=0,
$$
for any finite $N$. By the continuity of $X_1$ 
and $X_2$, we have that $\tau_N\to T$ a.s.~as 
$N\to \infty$. Therefore, sending $N\to \infty$, we arrive 
at \eqref{eq:pathwise-uniq}. 

In what follows, we consider $X_2-X_1$
and assume \eqref{eq:N-bound}. We have by linearity
\begin{align*}
	\d (X_2 - X_1) 
	& = \bk{u(t,X_2) - u(t,X_1)} \,\d t 
	+\frac14\left( \left(\sigma^2\right)'(t,X_2)
	-\left(\sigma^2\right)'(t,X_1)\right) \,\d t 
	\\ &\quad
	+\bk{ \sigma(t,X_2)-\sigma(t,X_1)}\, \d W.
\end{align*}
Set $Y: = \abs{X_2 - X_1}$. By the Tanaka formula, 
$$
\d Y = \sgn\left(X_2 - X_1\right) \,\d \left(X_2 - X_1\right)
+\frac12  \bk{\sigma(t,X_2)-\sigma(t,X_1)}^2\, \d L^0_{Y}(t).
$$
Since the local time $L^0_{Y}$ at $0$ of $Y$ is 
supported on the zero set of $X_2 - X_1$, which 
is a subset of the zero set of $\sigma(t,X_2)-\sigma(t,X_1)$, 
the local time correction term is zero.  
Set $\phi_\sigma(t): = \bk{\sigma(t,X_2) 
-\sigma(t,X_1)}/Y(t)$, which is a process 
uniformly bounded in absolute value by $\Lambda(t)$ 
of \eqref{eq:sigma_lipschitz}. Integrating in time yields
\begin{equation}\label{eq:difference_sde}
	\begin{aligned}
		Y(t)& = \int_0^t\sgn\bk{X_2(s)-X_1(s)}
		\int_{X_1(s)}^{X_2(s)} 
		\left( q(s,y)+\frac14\bk{\sigma^2}''(s,y) \right)
		\,\d y\, \d s
		\\&\quad
		+\int_0^t \phi_\sigma(s) Y(s)\,\d W(s),
	\end{aligned}
\end{equation}
where, in view of \eqref{eq:sigma_lipschitz} 
and \eqref{eq:N-bound}, the last 
term is a square-integrable martingale starting from zero; 
see \eqref{eq:stoch_oleinik} for the definition $q$.  
Making use of \eqref{eq:sigma_L2_incl} and taking the expectation, we obtain
\begin{align*}
	\Ex Y(t) & = \Ex \int_0^t 
	\sgn\bk{X_2(s)-X_1(s)}\int_{X_1(s)}^{X_2(s)} 
	q(s,y) \,\d y \,\d s 
	\\ & \quad 
	+\frac14\Ex \int_0^t \sgn(X_2(s) - X_1(s))
	\int_{X_1(s)}^{X_2(s)} 
	\bk{\sigma^2}''(s,y) \,\d y \,\d s 
	\\ & \le 
	\Ex \int_0^t Y^{1/2}(s) 
	\norm{q(s)}_{L^2(\Delta_s)} \,\d s
	+\frac14
	\norm{\left(\sigma^2\right)''(0)}_{L^\infty(\Omega \times \R)}
	\int_0^t \Ex Y(s) \,\d s
	\\ & \quad 
	+ \frac14\Ex\int_0^t  Y^{1/2}(s)
	\norm{\bk{\sigma^2}''(s)
	-\bk{\sigma^2}''(0)}_{L^2(\Delta_s)}\,\d s
	\\ & \le  
	\int_0^t \bk{\Ex Y(s)}^{1/2} 
	\Biggl[
	\bk{\Ex \norm{q(s)}_{L^2(\Delta_s)}^2}^{1/2} 
	\\ & \qquad \qquad \qquad \qquad \qquad
	+\frac14\bk{ \Ex\norm{\bk{\sigma^2}''(s)
	-\bk{\sigma^2}''(0)}_{L^2(\Delta_s)}^2}^{1/2}
	\Biggr]\,\d s
	\\ & \quad 
	+ \frac14 
	\norm{\left(\sigma^2\right)''(0)}_{L^\infty(\Omega \times \R)}
	\int_0^t \Ex Y(s) \,\d s,
\end{align*}
by the Cauchy--Schwarz inequality. 
Here, $\Delta_s$ denotes the (random) interval
$$
\Delta_s = \Bigl[X_1(s) \wedge X_2(s),\, 
X_1(s) \vee X_2(s)\Bigr].
$$

Taking the supremum over $t\in [0,\ep]$ on both sides gives
\begin{align*}
	&\sup_{t\in [0,\ep]}\Ex Y(t) \\
	&\le  \frac{\ep}{4} 
	\norm{\left(\sigma^2\right)''(0)}_{L^\infty(\Omega \times \R)}
	\sup_{t\in [0,\ep]}\Ex Y(t) 
	+ \ep \sup_{t\in [0,\ep] }\bk{\Ex Y(t)}^{1/2} 
	\\ & \qquad \times \sup_{t\in [0,\ep]}
	\left[\bk{\Ex \norm{q(t)}_{L^2(\Delta_t)}^2}^{1/2} 
	+ \frac14\bk{\Ex \norm{\bk{\sigma^2}''(t) 
	- \bk{\sigma^2}''(0)}_{L^2(\Delta_t)}^2}^{1/2}\right].
\end{align*}
Fix $\ep$ so small that 
$\frac{\ep}{4}\norm{\left(\sigma^2
\right)''(0)}_{L^\infty(\Omega \times \R)}\le \frac12$. 
The first term on the right-hand side can be 
absorbed by the term on the left-hand side. We then divide 
through by $\sup\limits_{t\in [0,\ep]} \bk{\Ex Y(t)}^{1/2}$ 
and square both sides, eventually arriving at
\begin{equation}\label{eq:small_time}
	\begin{aligned}
		\sup_{t\in [0,\ep] }\Ex Y(t)
		&  \le  8 \ep^2 \sup_{t\in [0,\ep]}
		\Ex \left[\norm{q(t)}_{L^2(\R)}^2 
		+\norm{\bk{\sigma^2}''(t) 
		-\bk{\sigma^2}''(0)}_{L^2(\R)}^2\right]
		\lesssim \eps^2.
	\end{aligned}
\end{equation}

The estimate \eqref{eq:small_time} allows us 
to control $\Ex Y(t)$ near $t = 0$. 
Using the one-sided bound \eqref{eq:stoch_oleinik}, which 
deteriorates near $t = 0$ for every $\omega \in \Omega$, 
in combination with the quadratic short-time estimate \eqref{eq:small_time}, we will next 
deduce a global estimate on the entire time 
interval $[0,T]$. 

Given the short-time estimate \eqref{eq:small_time}, 
we begin afresh from \eqref{eq:difference_sde}. 
Again let $\ep$ be so small that $\ep \norm{\left(\sigma^2
\right)''(0)}_{L^\infty(\Omega \times \R)}\le 2$, 
and $t>\eps$. We can then write the inequality
\begin{align*}
	Y(t)& = \int_0^\ep \sgn\bk{X_2(s)-X_1(s)}
	\int_{X_1(s)}^{X_2(s)} q(s,y)
	\,\d y \,\d s
	\\ &\qquad 	+\int_\ep^t \sgn\bk{X_2(s) - X_1(s)}
	\int_{X_1(s)}^{X_2(s)} q(s,y)	\,\d y \,\d s 
	\\ & \qquad	+ \frac14\int_0^t \sgn\bk{X_2(s) - X_1(s)} 
	\int_{X_1(s)}^{X_2(s)} \bk{\sigma^2}''(s,y)\,\d y\,\d s 
	\\ & \qquad +\int_0^t\phi_\sigma(s) Y(s)\,\d W(s)
	\\ & \le \int_0^t \eta_\ep(s)\,\d s
	+ \int_0^t Y(s)\,\d A_\ep(s)+M(t),
\end{align*}
where, for $t\in [0,T]$, 
\begin{align}
	M(t)&:= \int_0^t \phi_\sigma(s) 
	Y(s)\,\d W(s),\notag\\
	\eta_\ep(t) &:=
	\mathds{1}_{\{t\le \ep\}}
	\sgn\bk{X_2(t) - X_1(t)}
	\int_{X_1(t)}^{X_2(t)} \abs{q(t,y)}\,\d y,
	\notag\\
	\intertext{and,}\notag\\
	A_\ep(t) & := \int_0^t 
	\left[\mathds{1}_{\{s \ge \ep\}} K(s)
	+\Lambda(s) \right]\,\d s, 
	\label{eq:gronwall_functions3}
\end{align}
for $K(t)=K(\omega,t)$ defined 
in \eqref{eq:stoch_oleinik_g}, and because, from \eqref{eq:sigma_lip},
\begin{align*}
	\frac14 \sgn\bk{X_2(s) - X_1(s)} 	\int_{X_1(s)}^{X_2(s)} 
	\bk{\sigma^2}''(s,y)\,\d y 
	\le \frac14 Y(s) \Lambda(s).
\end{align*}
 
The adapted process $\eta_\eps$ is 
non-negative. Furthermore, using first the 
Cauchy--Schwarz inequality and then the short-time 
estimate \eqref{eq:small_time}, we have
\begin{align*}
	\Ex \int_0^t \eta_\eps(s)\,ds
	& \le  \int_0^\eps \bk{\Ex Y(s)}^{1/2} 
	\bk{\Ex \norm{q(s)}_{L^2(\Delta_s)}^2}^{1/2} \,\d s
	\lesssim \eps^2 \rho(\eps),
\end{align*}
where
\begin{align*}
	\rho(\eps)	
	& :=\bk{ \sup_{s\in [0,\eps]}
	\Ex \norm{q(s)}_{L^2(\Delta_s)}^2}^{1/2}.
\end{align*}
We will show that $\rho(\eps)=o(1)$ as $\eps\to0$.	
Furthermore, $A_\ep$ is a non-decreasing adapted process 
with $A(0) =0$. From \eqref{eq:log_divergence2} 
and \eqref{eq:sigma_exp},
\begin{align*}
	\Ex \exp\bk{\mu A_\ep(t) } &
	= \Ex \exp \bk{\mu \int_\ep^t K(s) \,d s
	+\mu\int_0^t \Lambda(s) \,\d s}
	\\ &\le \bk{\Ex \exp 
	\bk{2\mu \int_\ep^t K(s) \,d s }}^{1/2}
	\bk{\Ex \exp\bk{2\mu \int_0^t 
	\Lambda(s)\,\d s }}^{1/2}
	\\ &\le C_\mu \ep^{-2\mu},
\end{align*}
for a number $\mu$ such that 
$2\mu=p$, cf.~\eqref{eq:log_divergence2}.

Finally, by \eqref{eq:sigma_lipschitz} 
and \eqref{eq:N-bound}, $M$ is a (square-integrable) 
martingale with $M(0) = 0$. 

Hence, in view of Lemma~\ref{lem:stochastic_gronwall}, 
the stochastic Gronwall inequality
with $p = \frac12$ and a suitable 
$r\in \left(1/2,1\right)$, we arrive at
\begin{equation}\label{eq:gronwallbound1}
	\begin{aligned}
		\bk{\Ex \sup_{s \in [0,t]} Y^{1/2}(s)}^2
		&\le\bk{\frac{2r}{2r-1}}^2 
		\bk{\Ex \exp\Big(\frac{r}{1-r}
		A_\ep(t)\Big)}^{(1-r)/r} 
		\Ex\bk{\int_0^t \eta_\ep(s) \,\d t}
		\\ & 
		\overset{\eqref{eq:stoch_oleinik_g}}{\le} 
		C_r e^{C_\sigma (t - \ep)} 
		\bk{\ep^{-2r/(1 - r)}}^{(1-r)/r}
		\ep^2\rho(\eps)\lesssim \rho(\eps),
	\end{aligned}
\end{equation}
where $C_r$ is a constant depending 
only on $r$ and $C_\sigma$ is 
coming from \eqref{eq:sigma_init}.

Next we will show that the right-continuity 
condition \eqref{eq:u_one-sided_continuity} 
ensures that
\begin{equation}\label{eq:one-sided_continuity-qest}
	\lim_{\ep \to 0}\sup_{t\in [0,\ep]}
	\Ex \norm{q(t)}_{L^2(\Delta_t)}^2 = 0.
\end{equation}
Clearly, 
\begin{equation*}
	\sup_{t\in [0,\ep]}
	\Ex \norm{q(t)}_{L^2(\Delta_t)}^2 
	\le 2 \sup_{t\in [0,\ep]} 
	\Ex \norm{q(t)-q(0)}_{L^2(\Delta_t)}^2
	+2\sup_{t\in [0,\ep]} 
	\Ex \norm{q(0)}_{L^2(\Delta_t)}^2.
\end{equation*}
The first term on the right-hand side is bounded 
by $2\, \sup_{t\in [0,\ep]}
\Ex \abs{u(t)-u_0}_{\dot{H}^1(\R)}^2$, which 
tends to zero by \eqref{eq:u_one-sided_continuity}.  
Since $\abs{\Delta_t} = Y(t)$, 
\eqref{eq:small_time} implies 
$\mathds{1}_{\Delta_t} \to 0$, 
$\mathbb{P}$-almost surely, as $t \to 0$. 
We have $\Ex \norm{q(0)}_{L^2(\Delta_t)}^2
= \Ex \left(\mathds{1}_{\Delta_t} 
\norm{q(0)}_{L^2(\R)}^2\right) \to 0$ as $t\to 0$ by 
the dominated convergence theorem, 
since $q(0) \in L^2(\Omega \times \R)$. 
This proves \eqref{eq:one-sided_continuity-qest}.

Given \eqref{eq:one-sided_continuity-qest} 
and \eqref{eq:sigma_one-sided_tcont}, it follows 
that $\rho(\eps)=o(1)$ as $\eps\to 0$. 
As a result, we can send $\ep \to 0$ 
in \eqref{eq:gronwallbound1} 
to reach the conclusion that 
$\Ex \sup_{s \in [0,t]} Y^{1/2}(s)=0$, 
for any $t\in [0,T]$, which implies the 
desired result \eqref{eq:pathwise-uniq}.
\end{proof}

\begin{rem}
We point out that whilst the result above 
holds for $q(0) \in L^2(\R)$, that is, 
$q^2(0) \in L^1(\R)$, it fails for general $q(t)$ 
for which the right-continuity limit
$\lim_{t \downarrow 0} q^2(t)$ exists only 
in the sense of measures---but not 
in $L^1$ as required by \eqref{eq:u_one-sided_continuity}. 
An example comes from the deterministic Hunter--Saxton 
equation with an initial condition of the form 
$q^2(0)= \delta_0$. Although it is possible 
to define characteristics for this case, 
the characteristics emanating from $x = 0$ are not unique. 
The temporal continuity condition 
\eqref{eq:u_one-sided_continuity} is essential.
\end{rem}

\section{Existence of solution}\label{sec:existence}
In this section, we establish the existence of strong 
solutions for the SDE \eqref{eq:X_sde_a} by approximating 
\eqref{eq:X_sde_a} using a truncated coefficient in 
a way that allows us to apply a well-posedness
theorem of Krylov, reproduced below. We then show 
that the solutions to the approximating SDEs form 
a Cauchy sequence in an appropriate space, 
from which we recover a solution to our SDE. 

We begin by recalling Krylov's theorem for the 
well-posedness of SDEs with random 
coefficients \cite[Thm.~1.2]{MR1731794}.

\begin{thm}\cite{MR1731794}\label{thm:krylov_existence}
Let $\bS$ be a stochastic basis. 
Assume that for any $\omega \in \Omega$, $t \ge 0$, 
and $x \in \R^d$, we have $V(\omega,t,x) \in \R^{d \times d}$ 
and $b(\omega,t,x) \in \R^d$, and that ${V}$ and $b$ 
are continuous in $x$ for any $(\omega, t)$, 
and measurable in $(\omega, t)$. Moreover, assume
\begin{itemize}
	\item[(i)] boundedness: for any $T,\ell\in [0,\infty)$, 
	$\omega \in \Omega$, and any matrix norm $\norm{V}$, 
	$$
	\int_0^T \sup_{\abs{x}<\ell} 
	\bk{\abs{b(t,x)}+\norm{V(t,x)}^2}
	\,\d t<\infty.
	$$
	
	\item[(ii)] monotonicity: for all $t,\ell\in[0,\infty)$, 
	$x,y\in B_\ell(0)$, the ball with 
	radius $\ell$ and centred 
	at the origin, and $\omega \in \Omega$,
	$$
	2(x-y) \cdot \bigl(b(t,x)-b(t,y)\bigr) 
	+ \norm{V(t,x)-V(t,y)}^2 \le \tilde{K}(t,\ell) \abs{x-y}^2.
	$$
	
	\item[(iii)] coercivity: for all $t,\ell\in [0,\infty)$, 
	$x \in B_\ell(0)$, and $\omega \in \Omega$,
	$$
	2 x \cdot b(t,x) + \norm{V(t,x)}^2 
	\le \tilde{K}(t,1)\left(1 + \abs{x}^2\right),
	$$
\end{itemize}
where $\tilde{K}(t,\ell)$ is an adapted non-negative 
processes satisfying
\begin{align}\label{eq:K_bound}
	\int_0^T \tilde{K}(t,\ell)\, \d t<\infty, 	
	\quad 
	\text{for all $\omega \in \Omega$, $T,\ell\in [0,\infty)$}.
\end{align}
Let $X_0$ be an $\mathcal{F}_0$-measurable 
$\R^d$-valued random variable. Then the SDE
$$
\d X(t)=b(t,X(t))\,\d t+V(t, X(t)) \,\d W(t), 
\quad X(0)=X_0
$$ 
has a solution which is unique up 
to indistinguishability. Moreover,
\begin{equation}\label{eq:Krylov-L2est}
	\Ex \left( e^{-\alpha(t)} X^2(t)\right)
	\le x^2 +1, \qquad 
	\alpha(t) :=\int_0^t \tilde{K}(s,1)\, \d s.
\end{equation}
\end{thm}

\begin{rem}[Logarithmic divergence]\label{rem:log_div}
The monotonicity condition in 
Theorem \ref{thm:krylov_existence} 
can be viewed as a one-sided Lipschitz condition. 
In our motivating example, cf.~Remark~\ref{rem:example} 
and the one-sided gradient 
bound \eqref{eq:stoch_oleinik}, we have
$$
(x-y)\bk{u(t,x)-u(t,y)} 
\le  \abs{x-y}^2 \bk{ C 
+ \frac{e^{-\norm{\sigma'}_{L^\infty} W(t)}}
{\frac12\int_0^t e^{-\norm{\sigma'}_{L^\infty} W(s)}
\,\d s}}.
$$
Unfortunately, the factor multiplying $\abs{x-y}^2$ is not 
sufficiently well controlled at $t= 0$ to 
ensure \eqref{eq:K_bound}. There is the 
possibility of a logarithmic divergence in 
the temporal integral. As a result, 
Theorem \ref{thm:krylov_existence} 
does not apply to our problem.
\end{rem}

Next we introduce an approximate SDE by 
truncating the gradient $q = \pd_x u$. 
The reason for doing so is explained in 
Remark~\ref{rem:log_div}. The strong well-posedness 
of these approximate SDEs then follows 
from Theorem \ref{thm:krylov_existence}.

\begin{lem}\label{lem:existence_lemma_R}
Suppose $u \in L^p(\Omega;L^\infty([0,T];
\dot{H}^1(\R)))$ satisfies 
conditions \eqref{eq:stoch_oleinik_g} 
and \eqref{eq:u_one-sided_continuity}, and 
$\sigma$ satisfies \eqref{eq:sigma_lip}, 
\eqref{eq:sigma_lipschitz}--\eqref{eq:sigma_one-sided_tcont}.
Fix $R > 0$. Let $u_R$ be the process obtained 
from $q := \pd_x u$ by one-sided 
truncation at level $R$:
\begin{align}\label{eq:truncated_u}
	u_R(t,x) := \int_{-\infty}^x \vartheta_R(q(t,y))
	\,\d y, 
	\quad 
	\vartheta_R(q):= 
	\begin{cases}
		q, & \text{if $q \le R$},
		\\ 
		R, & \text{if $q> R$}. 
	\end{cases}
\end{align}
The SDE 
\begin{align}\label{eq:X_sde_R}
	\d X_R = u_R(t,X_R) \,\d t 
	+\frac14 \left(\sigma^2\right)'(t,X_R) \,\d t 
	+\sigma(t,X_R) \,\d W(t), \quad X(0) = x \in \R
\end{align}
has a unique strong solution.
\end{lem}

\begin{proof}
We take $b = u_R + \frac14 \left(\sigma^2\right)'
=u_R + \frac12 \sigma' \sigma$ 
and $V = \sigma$, on $\R^d$ with $d = 1$. 
The lemma follows from Theorem \ref{thm:krylov_existence}
once we have verified conditions (i), (ii), and (iii). 

By assumption, $\Ex \norm{u}_{L^\infty([0,T]\times\R)}^p
\lesssim \Ex \abs{u}_{L^\infty([0,T];\dot{H}^1(\R))}^p
\lesssim_p 1$ for all $p\in [1,\infty)$. Of course, the same 
bound holds for $u_R$: 
\begin{equation}\label{eq:uR_pmoment}
	\Ex \norm{u_R}_{L^\infty([0,T]\times \R)}^p 
	\lesssim_p 1.
\end{equation}
From this bound (with $p=1$),
$$
\sup_{\substack{t\in [0,T] \\ \abs{x}<\ell}} 
\abs{u_R(\omega,t,x)}<\infty,
\quad \text{for $\mathbb{P}$-a.e.~$\omega\in \Omega$}.
$$
The Lipschitz condition \eqref{eq:sigma_lip} 
and \eqref{eq:sigma_lipschitz}, 
\eqref{eq:sigma_origin} imply
$$
\int_0^T \sup_{\abs{x}<\ell} 
\left(\sigma^2\right)'(t,x)\,\d t 
\le \int_0^T 
\left( \abs{\left(\sigma^2\right)'(t,0)} 
+\Lambda(t) \ell\right)\,\d t 
< \infty. 
$$
Similarly, we have
$$
\int_0^T \sup_{\abs{x}<\ell} 
\sigma^2(t,x)\,\d t 
\le \int_0^T 
\bigl(\abs{\sigma(t,0)} + \Lambda(t)\ell\bigr)^2
\,\d t < \infty.
$$
Hence
\begin{equation*}
	\int_0^T \sup_{\abs{x}<\ell} \left(
	\abs{u_R(t,x)+\frac14 \left(\sigma^2\right)'(t,x)}
	+\abs{\sigma(t,x)}^2\right)\,\d t  < \infty,
\end{equation*}
which is (i).

For condition (ii), we have by 
\eqref{eq:sigma_lip} that
\begin{align*}
	&2(x-y) \,\bk{u_R(t,x)-u_R(t,y)  
	+\frac14\left(\sigma^2\right)'(t,x) 
	-\frac14\left(\sigma^2\right)'(t,y)}
	\\ & \qquad\qquad\qquad\qquad\qquad
	+\abs{\sigma(t,x)-\sigma(t,y)}^2
	\\ & \quad
	\le 2 \abs{x-y} 	
	\bk{\abs{\int_x^y \vartheta_R(q(t,z)) \,\d z} 
	+\frac14 \Lambda(t)\abs{x-y}}
	+\Lambda^2(t)\abs{x-y}^2
	\\ & \quad 
	\le \bk{2R+\Lambda(t)+\Lambda^2(t)}
	\abs{x-y}^2=: \tilde{K}_1(t) \abs{x-y}^2,
\end{align*}
and $\tilde{K}_1(t)$ is readily seen to 
satisfy \eqref{eq:K_bound} by \eqref{eq:sigma_lipschitz}.

Finally, condition (iii) is a result of
\begin{align*}
	& 2x \, u_R(t,x)
	+ \frac12  x \left(\sigma^2\right)'(t,x)
	+\sigma^2(t,x) 
	\\ & \quad
	\le 2\abs{x}  \norm{u_R(t)}_{L^\infty(\R)} 
	+\frac12 \abs{x}\left( \abs{\left(\sigma^2\right)'(t,0)}
	+\Lambda(t) \abs{x}\right)
	+\bigl(\abs{\sigma(t,0)}+\Lambda(t) \abs{x}\bigr)^2 
	\\ &\quad
	\le \left(\norm{u_R(t)}_{L^\infty(\R)}
	+ \abs{\left(\sigma^2\right)'(t,0)} 
	+ \frac12\Lambda(t)
	+2\sigma^2(t,0)
	+\Lambda^2(t)\right) \left(1+x^2\right)
	\\ &\quad
	=:\tilde{K}_2(t) \left(1 + x^2\right),
\end{align*}
where we have used \eqref{eq:sigma_lip}, 
\eqref{eq:sigma_lipschitz}, and \eqref{eq:sigma_origin}. 
By \eqref{eq:uR_pmoment}, \eqref{eq:sigma_lipschitz}, and 
\eqref{eq:sigma_origin}, it follows that $\tilde{K}_2$ 
satisfies \eqref{eq:K_bound}.

If we take $\tilde{K}(t,\ell)=\tilde{K}(t):=\tilde{K}_1(t)+\tilde{K}_2(t)$ (so $\tilde{K}$ is 
independent of $\ell$, but dependent on $R$), then 
all three conditions are verified.
\end{proof}

The next lemma supplies $R$-independent estimates 
for $X_R$ in $L^p(\Omega; C([0,T]))$ 
for any finite $p$. Note carefully that 
the $L^2$-estimate on $X_R(t)$ coming 
from Theorem \ref{thm:krylov_existence}, 
cf.~\eqref{eq:Krylov-L2est}, is useless 
because our $K$ depends on $R$.

\begin{lem}\label{lem:X_R_pmoment}
Let $X_R$ be the solution constructed in 
Lemma \ref{lem:existence_lemma_R}. 
Assume in addition that \eqref{eq:sigma_exp} and 
\eqref{eq:sigma_origin} hold. 
We have the uniform-in-$R$ bound
\begin{equation}\label{eq:XR-Lp}
	\Ex \sup_{t\in [0,T]} \abs{X}^p 
	\lesssim_{T,p} 
	\abs{x}^{4p} \lesssim_{x,T,p} 1, 
	\qquad p\in [1,\infty). 
\end{equation}
\end{lem}

\begin{proof}
We make frequent use of the following elementary 
inequalities, which hold for all $r\ge 2$ 
and $a,b,\epsilon>0$:
$$
a^{r-1}b\le \frac{\epsilon (r-1)}{r} a^r
+\frac{1}{\epsilon^{r-1}r} b^r,
\quad
a^{r-2}b^2 \le \frac{\epsilon (r-2)}{r} a^r
+\frac{2}{\epsilon^{(r-1)/2}r} b^r.
$$
By It\^o's formula, $\abs{X_R(t)}^{2p}=\abs{x}^{2p}
+I_1(t)+I_2(t)+I_3(t)+M(t)$, where
\begin{align*}
	&I_1(t)= 2p\int_0^t 
	\sgn\bk{X_R}\abs{X_R}^{2p-1}u_R(s,X_R) \,\d s,
	\\ & I_2(t)= \frac{p}{2} \int_0^t \sgn\bk{X_R} 
	\abs{X_R}^{2p-1} \left(\sigma^2\right)'(s,X_R)\,\d s,
	\\ & I_3(t)= p(2p - 1)
	\int_0^t \abs{X_R}^{2p-2}\sigma^2(s,X_R)\,\d s,
	\\ & 
	M(t)=2p\int_0^t \sgn\bk{X_R}
	\abs{X_R}^{2p-1}\sigma(s,X_R)\,\d W(s).
\end{align*}
Given \eqref{eq:sigma_lip}, we readily derive 
the bounds
\begin{align*}
	&I_1(t)  
	\le t \norm{u_R}^{2p}_{L^\infty([0,T]\times \R)}
	+\tilde{C}_p \int_0^t \abs{X_R}^{2p} \,\d s,
	\\ &
	I_2(t) 
	\le \tilde{C}_p 
	\int_0^t \left(\bk{1+\Lambda(s)} \abs{X_R}^{2p} 
	+ \abs{\left(\sigma^2\right)'(s,0)}^{2p}
	\right)\,\d s,
	\\ &
	I_3(t) 
	\le \tilde{C}_p\int_0^t 
	\bk{1+\Lambda^2(s)} \abs{X_R}^{2p} \,\d s
	+ \tilde{C}_p\int_0^t \abs{\sigma(s,0)}^{2p} \,\d s,
\end{align*}
for a constant $\tilde{C}_p$ depending only $p$. From this 
we obtain the inequality
\begin{align*}
	\abs{X_R(t)}^{2p}&\le \abs{x}^{2p}
	+t\norm{u_R}^{2p}_{L^\infty([0,T]\times\R)} 
	+C_p\int_0^t \abs{\left(\sigma^2\right)'(s,0)}^{2p}\,\d s\\
	&\quad+C_p\int_0^t \abs{\sigma(s,0)}^{2p} \,\d s
	+ C_p \int_0^t \bk{1+\Lambda^2(s)} 
	\abs{X_R(s)}^{2p} \,\d s +M(t),
\end{align*}
for another constant $C_p$ depending only $p$.

For any $N>0$, introduce the stopping time
$$
\tau_N:=\inf\left\{t\in [0,T]:\abs{X_R(t)}>N\right\}.
$$
By the continuity of $X_R$ we have 
that $\tau_N\to T$, $\mathbb{P}$-almost surely, 
as $N\to \infty$. Clearly, for $t\in [0,T]$,
\begin{align*}
	 \abs{X_R\left(t\wedge \tau_N\right)}^{2p} 
	&\le \abs{x}^{2p}
	+\left(t\wedge \tau_N\right)
	\norm{u_R}^{2p}_{L^\infty([0,T]\times\R)} 
	\\ & \quad
	+C_p\int_0^{t\wedge \tau_N} 
	\abs{\left(\sigma^2\right)'(s,0)}^{2p}\,\d s
	+C_p\int_0^{t\wedge \tau_N} 
	\abs{\sigma(s,0)}^{2p}\,\d s
	\\ &\quad
	+C_p \int_0^{t\wedge \tau_N} 
	\bk{1+\Lambda^2(s)} 
	\abs{X_R(s\wedge \tau_N)}^{2p} \,\d s 
	+M(t\wedge \tau_N),
\end{align*}
where $t\mapsto M(t\wedge \tau_N)$ 
is a (square-integrable) martingale starting from zero.  

Using the stochastic Gronwall inequality 
(Lemma \ref{lem:stochastic_gronwall} with 
exponents $\frac12$ and $\frac23$), 
\begin{equation*}
	\begin{aligned}
		&\bk{\Ex \sup_{t \in [0,T]} 
		\abs{X_R\left(t\wedge \tau_N\right)}^p}^{1/2} 
		\\ & \quad 
		\le \bk{\Ex\exp\bk{C_p\int_0^T
		\bk{1+\Lambda^2(t)}\,\d t} }^{1/2} 
		\\ &\qquad  
		\times  \Ex \bk{\abs{x}^{2p}
		+T\norm{u_R}_{L^\infty([0,T]\times \R)}^{2p} 
		+\int_0^T\abs{\left(\sigma^2\right)'(s,0)}^{2p}
		\,\d s+\int_0^T \abs{\sigma(t,0)}^{2p}\,\d t}.
	\end{aligned}
\end{equation*}
Given \eqref{eq:sigma_origin} 
and \eqref{eq:uR_pmoment}, we conclude that
$$
\Ex \sup_{t \in [0,T]}
\abs{X_R\left(t\wedge \tau_N\right)}^p 
\lesssim_{x,T,p} 1.
$$
Finally, sending $N\to \infty$, we 
arrive at \eqref{eq:XR-Lp}.
\end{proof}

To show that $\{X_R\}$ is a Cauchy 
sequence, we will require some compactness properties
of $u_R$ as $R \to \infty$. Since $u_R$ is constructed from
$u$ in an explicit manner, this is not 
difficult to establish:

\begin{lem}\label{lem:uR_convergence}
Suppose $u \in L^p(\Omega;L^\infty([0,T];
\dot{H}^1(\R)))$, for $p\in [1,\infty)$. 
Let $u_R$ be defined by the 
construction \eqref{eq:truncated_u}. 
We have the convergence
\begin{align}\label{eq:uR_limit1}
	\Ex \sup_{t\in [0,T]}
	\abs{u_R(t)-u(t)}_{\dot{H}^1(\R)}^2 
	\overset{R \to \infty}{\longrightarrow} 0.
\end{align}
Moreover, for any finite $p\ge 1$,
\begin{align}\label{eq:uR_limit2}
	\Ex \norm{u_R-u}_{L^\infty([0,T]\times \R)}^p 
	\overset{R \to \infty}{\longrightarrow} 0.
\end{align}
\end{lem}

\begin{proof}
We have that
$\abs{u_R(t)-u(t)}_{\dot{H}^1(\R))}^2$ equals
\begin{align*}
	I_R(t):=\int_\R \abs{\vartheta_R(q(t,y))-q(t,y)}^2\,\d y 
	&= \int_\R \abs{R - q(t,y)}^2 \mathds{1}_{\{q(t,y)>R\}}\,\d y 
	\\ & \le 4 \int_\R \abs{q(t,y)}^2
	\mathds{1}_{\{q(t,y)>R\}}\,\d y.
\end{align*}
Since $\Ex \norm{q}_{L^\infty([0,T];L^2(\R))}^2 \lesssim 1$ 
by assumption, we find that $\Ex I_R(t)$ tends 
to zero as $R \to 0$, uniformly in $t \in [0,T]$; 
hence \eqref{eq:uR_limit1} holds. We also have
\begin{align*}
	\tilde I_R(t):=\abs{u_R(t) - u(t)} 
	& = \abs{\int_{-\infty}^x \bigl(
	\vartheta_R(q(t,y))-q(t,y) \bigr) \,\d y} 
	\\ & \le  \int_{-\infty}^x \abs{R - q(t,y)} 
	\mathds{1}_{\{q(t,y)>R\}} \,\d y 
	\le \frac{1}{R}\norm{q(t)}_{L^2(\R)}^2.
\end{align*}
By assumption, for all $p\ge 1$ we 
have $q\in L^p(\Omega; L^\infty([0,T]; L^2(\R)))$ 
and therefore $\Ex \left(\sup_{t \in [0,T]}
\tilde I_R(t)\right)^p 
\overset{R \to \infty}{\longrightarrow} 0$. 
This proves the claim \eqref{eq:uR_limit2}.
\end{proof}

The next result, which is the main contribution of this 
section, reveals that $\{X_R\}$ is a 
Cauchy sequence in $L^{1/2}(\Omega;C([0,T]))$.

\begin{prop}\label{prop:cauchy_seq}
Under the assumptions of Lemma \ref{lem:existence_lemma_R}, 
suppose in addition that \eqref{eq:sigma_exp} is true 
and also that \eqref{eq:sigma_origin} 
holds with $p=2$. The solutions $X_R$ 
to \eqref{eq:X_sde_R}, which satisfy 
the $R$-independent bound $\Ex\sup_{t\in [0,T]}
\abs{X_R(t)}\lesssim_{T,x} 1$ 
(cf.~Lemma \ref{lem:X_R_pmoment}), 
form a sequence $\{X_R\}$ that is Cauchy 
in $L^{1/2}(\Omega;C([0,T]))$.
\end{prop}

\begin{proof}
For $N,R,R'>0$, define
$$
\tau_N^{R,R'}:=\inf\left\{t\in[0,T]:
\abs{X_R(t)}>N \mbox{ or }
\abs{X_{R'}(t)}>N \right\}.
$$
Replace $X_R$ by $\tilde{X}_R(t)
:=X_R\left(t\wedge \tau_N^{R,R'}\right)$,
which satisfies $\abs{\tilde{X}_R(t)}\le N$ for all
$t\in \left[0,\tau_N^{R,R'}\right]$.
The SDE for $\tilde{X}_R$ becomes
\begin{align*}
\tilde{X}_R(t)
& = x+\int_0^t u_R\left(s,\tilde{X}_R(s)\right)
\,\d s +\frac14 \int_0^t
\left(\sigma^2\right)'\left(s,\tilde{X}_R(s)\right)
\,\d s \\ &\quad +\int_0^t
\sigma\left(s,\tilde{X}_R(s)\right)\, \d W(s),
\quad t\in \left[0,\tau_N^{R,R'}\right].
\end{align*}
Applying the upcoming argument to
$\tilde{X}_R-\tilde{X}_{R'}$ on
the time interval $\left[0,\tau_N^{R,R'}\right]$,
where $\tilde{X}_{R'}(\cdot)
:=X_{R'}\left(\cdot\wedge \tau_N^{R,R'}\right)$,
we deduce that for any $\delta>0$ there
exists $R_0=R_0(\delta)$ such that, for
all $t\in [0,T]$,
\begin{equation*}
\Ex \sup_{s \in [0,t]}
\abs{X_R\left(s\wedge \tau_N^{R,R'}\right)
-X_{R'}\left(s\wedge \tau_N^{R,R'}\right)}^{1/2}
<\delta, \quad \text{for all $R,R'\ge R_0$},
\end{equation*}
see \eqref{eq:Cauchy_final}. To conclude from this,
one notices that $\tau_N^{R,R'}\to T$ as $N\to \infty$,
uniformly in $R,R'$. Indeed, the $R$-independent bound (Lemma \ref{lem:X_R_pmoment})
$\Ex\sup\limits_{t\in [0,T]} \abs{X_R}\lesssim 1$ implies
$$
\Px\left(\tau_N^{R,R'}<T\right)
\le \Px\left(\sup_{t\in \left[0,\tau_N^R\right]}
\abs{X_R(t)}\ge N,\, \tau_N^R<T\right)
\le \frac{1}{N} \Ex \sup_{t\in [0,T]}
\abs{X_R(t)} \to 0,
$$
as $N\to \infty$, uniformly in $R$. Hence,
$\tau_N^{R,R'}\to T$ as $N\to \infty$,
uniformly in $R,R'$.

Given the preceding discussion, in what follows, 
there is no loss of generality in assuming that 
\begin{equation}\label{eq:N-bound2}
	\abs{X_R(t)}, \abs{X_{R'}(t)}\le N, 
	\quad \text{for all $t\in [0,T]$},
\end{equation}
for some given $N>0$, when seeking to establish that
$$
Y(t)=Y_{R,R'}(t):=\abs{X_R(t) - X_{R'}(t)}, 
\qquad R,R'\in [0,\infty),
$$
satisfies the Cauchy property \eqref{eq:Cauchy_final}. 
The Tanaka formula gives 
\begin{equation}\label{eq:XR_tanaka}
	\begin{aligned}
		Y(t)&= \int_0^t \sgn(X_R - X_{R'}) 
		\bk{u_R(s,X_R)-u_{R}(s,X_{R'})} \,\d s 
		\\ & \quad
		+ \int_0^t \sgn(X_R - X_{R'}) 
		\bk{u_{R}(s,X_{R'}) - u_{R'}(s,X_{R'})}\,\d s 
		\\ &\quad
		+ \frac14 \int_0^t \sgn(X_R-X_{R'}) 
		\bk{ \left(\sigma^2\right)'(s,X_R)
		- \left(\sigma^2\right)'(s,X_{R'})}\,\d s  
		\\ &\quad
		+\int_0^t \sgn(X_R-X_{R'})
		\bk{ \sigma(s,X_R)-\sigma(s,X_{R'}) }\,\d W(s).
	\end{aligned}
\end{equation}
This is very similar to \eqref{eq:difference_sde}, 
except for the difference $u_R(s,X_{R'})
-u_{R'}(s,X_{R'}) $.

First we seek to estimate $Y(t)$ over a short 
time period $t \in [0,\ep]$.  In \eqref{eq:XR_tanaka}, as 
in the previous section, we write 
\begin{align*}
	&\int_0^t \sgn\bk{X_R-X_{R'}}
	\bk{u_R(s,X_R)-u_{R}(s,X_{R'}) }\,\d s 
	\\ &\quad\qquad
	+ \frac14 \int_0^t \sgn\bk{X_R-X_{R'}} 
	\bk{ \left(\sigma^2\right)'(s,X_R)
	- \left(\sigma^2\right)'(s,X_{R'})}\,\d s
	\\ & \quad 
	= \int_0^t \sgn\bk{X_R-X_{R'}} 
	\int_{X_{R'}(s)}^{X_R(s)}
	\vartheta_R(q(s,y))\,\d y\,\d s 
	\\ & \qquad
	+ \frac14 \int_0^t \sgn\bk{X_R-X_{R'}}  
	\int_{X_{R'}(s)}^{X_R(s)}\bk{\bk{\sigma^2}''(s,y) 
	-\bk{\sigma^2}''(0,y)}\,\d y\,\d s 
	\\ & \qquad
	+ \frac14 \int_0^t 
	\sgn\bk{X_R-X_{R'}}\int_{X_{R'}(s)}^{X_R(s)}
	\bk{\sigma^2}''(0,y)\,\d y\,\d s.
\end{align*}
Estimating by the Cauchy--Schwarz inequality,
\begin{equation}\label{eq:Yestimate1}
	\begin{aligned}
		Y(t) & \le   \int_0^t \sgn\bk{X_R - X_{R'}}
		\bk{u_R(s,X_{R'})-u_{R'}(s,X_{R'})}\,\d s
		\\ &\quad 
		+\int_0^t  Y^{1/2}(s) 
		\norm{q(s)}_{L^2(\Delta_s)}\,d s
		\\ &\quad
		+ \frac14  \int_0^t  Y^{1/2}(s) 
		\norm{ \bk{\sigma^2}''(s)
		-\bk{\sigma^2}''(0)}_{L^2(\Delta_s)}\,\d s
		\\ & \quad
		+\frac14 \int_0^t 
		\norm{\bk{\sigma^2}''(0)}_{L^\infty(\Omega \times \R)} 
		Y(s) \,\d s
		+\int_0^t \phi_\sigma(s) Y(s)\,\d W(s),
	\end{aligned}
\end{equation}
where $\phi_\sigma(t) := \bk{ \sigma(t,X_R)
-\sigma(t,X_{R'}) }/Y(t)$ is a 
process bounded in absolute value by $\Lambda(t)$ of \eqref{eq:sigma_lipschitz}. Here, $\Delta_s$ denotes 
the (random) interval
$$
\Delta_s = \Bigl[X_R(s) \wedge X_{R'}(s),\, 
X_R(s) \vee X_{R'}(s)\Bigr].
$$

Given \eqref{eq:N-bound2}, 
the last term in \eqref{eq:Yestimate1} is 
a square-integrable martingale starting from zero. 
Taking the expectation, and estimating as in the proof of  
Theorem~\ref{thm:pathwise_uniqueness},
\begin{align*}
	\Ex Y(t) & \le   
	t \Ex \norm{u_R-u_{R'}}_{L^\infty([0,t]\times \R)}
	\\ & \quad
	+ \int_0^t \bk{\Ex Y(s)}^{1/2} 
	\bk{\Ex\norm{q(s)}_{L^2(\Delta_s)}^2}^{1/2}\,d s
	\\ &\quad
	+ \frac14  \int_0^t \bk{\Ex Y(s)}^{1/2} 
	\bk{\Ex \norm{\bk{\sigma^2}''(s) 
	-\bk{\sigma^2}''(0)}_{L^2(\Delta_s)}^2}^{1/2}\,\d s
	\\ & \quad
	+\frac14 \int_0^t 
	\norm{\bk{\sigma^2}''(0)}_{L^\infty(\Omega \times \R)}
	\Ex Y(s) \,\d s.
\end{align*}

Taking the supremum over $t\in [0,\ep]$, and 
applying Young's inequality (in the form 
$ab = \left(\frac{1}{\sqrt{2\eps}}a\right) 
\left(\sqrt{2\eps}b\right)
\le \frac{1}{4\eps}a^2+\eps b^2$), we find
\begin{align*}
	\sup_{t \in [0,\ep]}\Ex Y(t) 
	& \le \ep \Ex \norm{u_R-u_{R'}}_{L^\infty([0,\ep]\times \R)}
	\\ &\quad 
	+\frac14 \sup_{t \in [0,\ep]}\Ex Y(t) 
	+\ep^2\sup_{t \in [0,\ep]}\Ex\norm{q(s)}_{L^2(\Delta_s)}^2
	\\ &\quad
	+\frac{1}{16} \sup_{t \in [0,\ep]}\Ex Y(s) 
	+\frac{\ep^2}{4} \sup_{t \in [0,\ep]}
	\Ex\norm{ \bk{\sigma^2}''(s) 
	-\bk{\sigma^2}''(0)}_{L^2(\Delta_s)}^2
	\\ & \quad
	+\frac{\ep}{4}
	\norm{\bk{\sigma^2}''(0)}_{L^\infty(\Omega \times \R)} 
	\sup_{t \in [0,\ep]}\Ex Y(t).
\end{align*}

In what follows, we fix $\ep$ so small that 
$\frac14+\frac{1}{16}+\frac{\ep}{4}
\norm{\bk{\sigma^2}''(0)}_{L^\infty(\Omega \times \R)}
\le \frac12$. Since $\ep$ and $R, R'$ are 
independent parameters, given \eqref{eq:uR_limit2} of Lemma~\ref{lem:uR_convergence}, 
we can take $R_0=R_0(\eps)$ so large that
\begin{align}
	\Ex \norm{u_{R}-u_{R'}}_{L^\infty([0,\ep]\times \R)}
	&\le \Ex\norm{u_{R}-u_{R'}}_{L^\infty([0,T]\times \R)} 
	\notag \\
	&\le \ep^{5/2}=\eps^2 o(1), 
	\quad \text{as $\eps\to0$}, 
	\label{eq:ep_squared3}
\end{align}
for all $R,R'\ge R_0(\eps)$. This gives us
\begin{align*}
	\sup_{t \in [0,\ep]}\Ex Y(t) 
	&\le 2\ep^2\Biggl(\ep^{3/2}
	+\sup_{t \in [0,\ep]}\Ex\norm{q(s)}_{L^2(\R)}^2
	\\ & \qquad \qquad \qquad \quad
	+\sup_{t \in [0,\ep]}\Ex \norm{\bk{\sigma^2}''(s) 
	-\bk{\sigma^2}''(0)}_{L^2(\R)}^2\Biggr).
\end{align*}

Importantly, from \eqref{eq:sigma_one-sided_tcont} and
\eqref{eq:one-sided_continuity-qest} we conclude that
\begin{equation}\label{eq:ep_squared2}
	\sup_{t\in [0,\ep]}\Ex Y(t)
	=\ep^2 o(1), \quad \text{as $\eps\to0$}.
\end{equation} 

As in the proof of Theorem \ref{thm:pathwise_uniqueness}, 
we estimate $Y$ again (this time on the entire 
time interval $[0,T]$). From \eqref{eq:Yestimate1}, 
we arrive at the integral inequality
\begin{equation*}
	Y(t) \le \int_0^t \eta(s) \,\d s
	+\int_0^t Y(s) \,\d A(s) + M(t),
\end{equation*}
where, for $t \in [0,T]$,
\begin{align*}
	M(t) &:=\int_0^t \phi_\sigma(s) Y(s)\,\d W, \\
	\eta(t) & := \mathds{1}_{\{t \le \ep\}} 
	Y^{1/2}(s) \norm{q(s)}_{L^2(\Delta_s)}
	+T\norm{u_R-u_{R'}}_{L^\infty([0,T]\times \R)},
\intertext{and, as in \eqref{eq:gronwall_functions3},}
	A(t)&:=\int_0^t\bk{\mathds{1}_{\{s\ge \ep\}} 
	K(s)+\Lambda(s) }\,\d s.
\end{align*}

Since we have not assumed an 
exponential moment bound for the difference 
$\norm{u_R-u_{R'}}_{L^\infty\bk{[0,T] \times \R}}$, it 
becomes imperative to include this term as 
a part of $\eta$ and not $A$. 
The process $\eta$ is non-negative and, 
by \eqref{eq:sigma_one-sided_tcont}, \eqref{eq:ep_squared2} 
and \eqref{eq:ep_squared3}, is controlled thus:
\begin{equation}\label{eq:est-of-eta}
	\Ex \int_0^t \eta(s)\,\d s 
	= \ep^2 o(1), 
	\quad \text{as $\eps\to0$}.
\end{equation}

Now we apply Lemma~\ref{lem:stochastic_gronwall},  
the stochastic Gronwall inequality
with $p = \frac12$ and a 
suitable $r\in \left(\frac12,1\right)$. 
In view of \eqref{eq:stoch_oleinik_g} 
and \eqref{eq:est-of-eta}, 
\begin{equation*}
	\bk{\Ex \sup_{s \in [0,t]} 
	Y^{1/2}(t)}^{2}
	\le C_r e^{C_\sigma (t-\ep)} 
	\ep^{-2}
	\ep^2 o(1)
	= o(1), \quad \text{as $\eps\to0$}.
\end{equation*}

Therefore, given any $\delta>0$, we can find 
$\eps=\eps(\delta)$ and $R_0=R_0(\delta):=
R_{\ep(\delta)}\vee R'_{\ep(\delta)}$ 
such that, for all $t\in [0,T]$, 
\begin{equation}\label{eq:Cauchy_final}
	\Ex \sup_{s \in [0,t]} 
	Y^{1/2}_{R,R'}(s) <\delta, 
	\quad \forall R,R'\ge R_0.
\end{equation}
This concludes the proof of the proposition.
\end{proof}

Proposition \ref{prop:cauchy_seq} 
implies convergence in probability.

\begin{lem}\label{lem:conv-in-prob}
Under the assumptions of Lemma \ref{lem:existence_lemma_R}, 
suppose in addition that \eqref{eq:sigma_exp} is true and 
also that \eqref{eq:sigma_origin} 
holds with $p=2$. Then there exists a 
$\mathbb{P}$-almost surely continuous and 
$\{\mathcal{F}_t\}_{t \ge 0}$-adapted stochastic 
processes $X:\Omega\times [0,\infty)\to \R$ such that
\begin{equation}\label{eq:conv-in-prob}
	\lim_{R\to \infty} \Px\left( 
	\sup_{t\in[0,T]}\abs{X_R(t)-X(t)}
	>\ep\right)=0,
\end{equation}
for all $\ep>0$, for all finite $T>0$.
\end{lem}

\begin{proof}
By Chebyshev's inequality 
and Proposition \ref{prop:cauchy_seq}, we obtain 
\begin{align*}
	&\Px\left( 
	\sup_{t\in[0,T]}\abs{X_{R}(t)-X_{R'}(t)}
	>\ep\right)\le \frac{1}{\sqrt{\varepsilon}}
	\Ex \sup_{t \in [0,T]} \abs{X_{R}(t)-X_{R'}(t)}^{1/2}
	\overset{R,R' \to \infty}{\longrightarrow} 0,
\end{align*}
so that $\{X_R\}$ is a Cauchy sequence in the 
space of continuous processes with 
respect to locally (in $t$) uniform 
convergence in probability. Since this 
space is complete, the lemma follows.  
\end{proof}

It remains to identify the limit $X$ as a 
solution to the original SDE \eqref{eq:X_sde_a}.

\begin{thm}[Existence of solution]\label{thm:existence}
Under the assumptions of Theorem~\ref{thm:main}, there 
exists a strong solution $X$ to the SDE \eqref{eq:X_sde_a}.
\end{thm}

\begin{proof}
Fix a finite number $T>0$. 
By Lemma \ref{lem:existence_lemma_R}, 
there exists a unique strong solution $X_R$ 
to the SDE \eqref{eq:X_sde_R}, such that
$$
X_R(t)=x+\int_0^t u_R (s,X_R)\,\d s 
+\frac14\int_0^t \left(\sigma^2\right)'(s,X_R)\,\d s 
+\int_0^t \sigma(s,X_R)\,\d W(s).
$$
Let $X$ be the limit process constructed in 
Lemma \ref{lem:conv-in-prob}. Then
\begin{align*}
	& I(t):=X-x-\int_0^t u (s,X)\,\d s 
	-\frac14\int_0^t \left(\sigma^2\right)'(s,X)\,\d s 
	-\int_0^t \sigma(s,X)\,\d W(s)
	\\ & \qquad 
	= I^{(1)}_R(t)+I^{(2)}_R(t)
	+ I^{(3)}_R(t)+M_R(t),
\end{align*}
where
\begin{align*}	
	& I^{(1)}_R(t)=X(t)-X_R(t), 
	\quad
	I^{(2)}_R(t)=\int_0^t \left(u_R(s,X_R)- u(s,X)
	\right)\,\d s, 
	\\ & 
	I^{(3)}_R(t)=\frac14 \int_0^t 
	\left(
	\left(\sigma^2\right)'(s,X_R) 
	-\left(\sigma^2\right)'(s,X)\right)\,\d s,
	\\ &
	M_R(t)=\int_0^t \bigl(\sigma(s,X_R)
	-\sigma(s,X)\bigr)\,\d W(s). 
\end{align*}
Because of the path continuity of $X$, it is enough
to prove that $I(t)=0$ $\mathbb{P}$-almost 
surely, for any fixed $t\in [0,T]$. 
To this end, we will verify that
$$
I^{(1)}_R(t),\, I^{(2)}_R(t),\, I^{(3)}_R(t), \, M_R(t) 
\overset{R \to \infty}{\longrightarrow} 0, 
\quad \text{$\mathbb{P}$-a.s.},
$$ 
at least for some subsequence 
$R_n\to 0$ as $n\to \infty$.

Since convergence in probability, cf.~\eqref{eq:conv-in-prob}, 
implies almost sure convergence along a subsequence, we have
\begin{equation}\label{eq:conv-as}
	\sup_{t\in[0,T]}\abs{X_{R_n}(t)-X(t)}
	\overset{n\to\infty}{\longrightarrow} 0, 
	\quad  \text{$\mathbb{P}$-a.s.},
\end{equation}
which implies that $I^{(1)}_{R_n}(t)\to 0$, 
$\mathbb{P}$-almost surely, as $n\to \infty$.

Given \eqref{eq:uR_limit2}, 
we have that
$$
\norm{u_{R_n}-u}_{L^\infty([0,T]\times \R)}
\overset{n\to\infty}{\longrightarrow} 0,
\quad  \text{$\mathbb{P}$-a.s.}
$$
Using this and the $\mathbb{P}$-almost sure bound
$\norm{q}_{L^\infty([0,T];L^2(\R))}^2<\infty$, 
we obtain
\begin{align*}
	\abs{I^{(2)}_{R_n}(t)}	
	&\le \abs{\int_0^t\int_{X_{R_n}}^{X} 
	\vartheta_{R_n}(q(s,y))\,\d y\,\d s} 
	+\abs{\int_0^t  \left(u_{R_n}(s,X)-u(s,X)\right)\,\d s}
	\\ & \le T\bk{\sup_{s\in [0,T]}
	\abs{X_{R_n}(s)-X(s)}}^{1/2}
	\bk{\norm{q}_{L^\infty([0,T];L^2(\R))}^2}^{1/2}
	 \\ & \quad\qquad\qquad 
	+T\norm{u_{R_n}-u}_{L^\infty([0,T]\times \R)}
	\overset{n \to \infty}{\longrightarrow} 0,
	\quad  \text{$\mathbb{P}$-a.s.}
\end{align*}

By \eqref{eq:sigma_lip}, \eqref{eq:sigma_lipschitz}, 
and \eqref{eq:conv-as},
\begin{align*}
	\abs{I^{(3)}_{R_n}(t)}
	& \le \int_0^T \Lambda(s)
	\abs{X_{R_n}(s)-X(s)}\,\d s
	\\ & \le 
	\Big({\int_0^T \Lambda(s)\,\d s}\Big)
	\sup_{s\in [0,T]} \abs{X_{R_n}(s)-X(s)}
	\overset{n\to\infty}{\longrightarrow} 0, 
	\quad  \text{$\mathbb{P}$-a.s.}
\end{align*}

Recall the bound $\Ex\sup_{t\in [0,T]}
\abs{X_R(t)}\lesssim_{T,x} 1$, which holds 
uniformly in $R$ (here we need \eqref{eq:sigma_exp} 
and \eqref{eq:sigma_origin} 
with $p=2$). Set $S(t):=\sup_{n\in \N} 
\sup_{s\in [0,t]}\abs{X_{R_n}(s)}$, which 
is bounded, $\mathbb{P}$-almost surely. 
For $N\in [0,\infty)$, introduce the stopping time
$$
\tau_N := \inf\left\{t\in [0,T]:S(t)>N \right\}.
$$
Clearly, $\Px\bigl(\tau_N<t\bigr)\to 0$ as $N\to \infty$. 
By \eqref{eq:sigma_lip}, \eqref{eq:sigma_lipschitz}, 
\eqref{eq:conv-as} and the dominated 
convergence theorem,
\begin{align*}
	&\Ex \abs{\int_0^{t\wedge\tau_N}
	\bigl(\sigma\left(s,X_{R_n}(s)\right)
	-\sigma\left(s,X(s)\right)\bigr) \, \d W(s)}^2
	\\ & \qquad = \Ex \int_0^t
	\mathds{1}_{\left[0,\tau_N\right]}(s)
	\abs{\sigma\left(s,X_{R_n}(s)\right)
	-\sigma\left(s,X(s)\right)}^2 \, \d s
	\overset{n\to\infty}{\longrightarrow} 0.
\end{align*}
As a result, for any $\ep>0$,
\begin{align*}
	& \Px\bigl(\abs{M_{R_n}(t)}>\ep\bigr)
	\le \Px\bigl( \abs{M_{R_n}(t)}>\ep, \tau_N \ge t\bigr)
	+\Px\bigl(\tau_N<t\bigr)
	\\ & \qquad \le 
	\frac{1}{\ep^2}\Ex \abs{\int_0^{t}
	\bigl(\sigma\left(s,X_{R_n}(s)\right)
	-\sigma\left(s,X(s)\right)\bigr) \, \d W(s)}^2
	+\Px\bigl(\tau_N<t\bigr).
\end{align*}
Sending first $n\to \infty$ and then $N\to \infty$, we 
conclude that $M_{R_n}(t)
\overset{n\to\infty}{\longrightarrow}0$ 
in probability, and therefore, along a further 
subsequence (not relabelled),
$$
M_{R_n}(t)
\overset{n\to\infty}{\longrightarrow}0,
\quad  \text{$\mathbb{P}$-a.s.}
$$
This completes the proof that $X$ is a solution of 
the SDE \eqref{eq:X_sde_a}.
\end{proof}


\bibliographystyle{plain}

\end{document}